\documentclass[12pt,halfline,a4paper]{article}
\usepackage{enumerate}
\usepackage{amsfonts}
\usepackage{amsthm,amssymb,amsmath}
\usepackage{mathrsfs,MnSymbol,setspace, xcolor}
\usepackage[colorlinks=true, linkcolor = blue, citecolor = cyan]{hyperref}
\usepackage[top=3cm,right=3cm,bottom=3cm,left=3cm]{geometry}                 
\usepackage{amsfonts}
\usepackage{yfonts}
\usepackage{verbatim}
\usepackage{graphicx}
\usepackage{tikz}
\usetikzlibrary{arrows}

\usepackage{soul}

\newtheorem{lemma}{\textbf{Lemma}}[section]

\newtheorem{theorem}{\textbf{Theorem}}[section]
\newenvironment{definition}{\textbf{Definition.}}{\hfill$\blacktriangleleft$}

\newcounter{nremark}
\newtheorem{remark}{\refstepcounter{nremark}\textbf{Remark}}{}

\numberwithin{equation}{section}

\author{Doratossadat Dastgheib\footnote{d\_dastgheib@sbu.ac.ir} , Hadi Farahani\footnote{h$_-$farahani@sbu.ac.ir}}
\title{A proof for completeness of \L ukasiewicz logic}

\begin{document}
	\maketitle
	\begin{abstract}
		 In this paper we give a new proof for the completeness of infinite valued propositional \L ukasiewicz logic introduced by \L ukasiewicz and Tarski in 1930. Our approach employs a Hilbert-style proof that relies on the concept of maximal consistent extensions, and unlike classical logic, in this context, the maximal extensions are not required to include all formulas or their negations.
		 To illustrate this point, we provide examples of such formulas.  
	\end{abstract}
	
	
	\section{Introduction}
	
\L ukasiewicz logic is one of non-classical logical fuzzy logics that allows to have more than two truth values \cite{Klement2000, Urquhart2001}.
The first proof of the completeness of Łukasiewicz logic was established by Rose and Rosser in \cite{Rose-Rosser1958} using a syntactic approach. Later, in \cite{change1959}, Chang presented an alternative proof based on algebraic methods. His approach demonstrated that any MV-algebra could be decomposed into a product of linearly ordered MV-algebras, which are simpler to work with. By utilizing results from the theory of ordered abelian groups, Chang showed that any formula valid in all linearly ordered MV-algebras could be derived from the axioms and rules of \L ukasiewicz logic. Additionally, in \cite{Botur2015}, another proof was proposed, demonstrating how to embed any finite partial subalgebra of a linearly ordered MV-algebra into $\mathbb{Q}\cap [0,1]$.
\\
In this paper, we present a simple Hilbert-style proof for the completeness of Łukasiewicz logic.

	\section{\L{ukasiewicz Logic}} \label{general_belief}
	In this section, we first remind the language of  \L ukasiewicz propositional logic. Then we review its semantics and then the axioms and also theorems. At end we provide the new proof of completeness.
	\subsection{Syntax and Semantics}
	\begin{definition}
		The \textit{language of propositional \L ukasiewicz logic} $\textgoth{L}$, is defined using the following BNF:
		$$\varphi::= p \;|\; \neg \varphi\;|\; \varphi \, \& \, \varphi \;|\; \varphi \rightarrow \varphi $$
		where $p\in\mathcal{P}$ and $\mathcal{P}$ is an enumerable  set of propositions containing $\perp$. The other connectives $\veebar$, $\wedge$ and $\vee$ are defined as in conventional format:
		$$\begin{array}{ll}
			\varphi \veebar \psi = \neg \varphi \rightarrow \psi, &
			\varphi \wedge \psi = \varphi \,\&\, (\varphi \rightarrow \psi),\\
			\multicolumn{2}{l}{\varphi \vee \psi = ((\varphi \rightarrow \psi) \rightarrow \psi ) \wedge ((\psi \rightarrow \varphi)\rightarrow \varphi).}
		\end{array}$$
	In this paper we also use the symbol $\leftrightarrow$  which is equivalence connective with similar definition as Definition 3.1.7 in \cite{Hajek1998}.
	\end{definition}

	In the following, we define the semantics using valuation function.\\
	\begin{definition}
		A 
		\textit{valuation}  is a function $\pi: \mathcal{P} \rightarrow [0,1]$ that assigns a  value ranges in the interval [0,1] to each proposition.
		Let $\pi$ be a valuation function. 
		The valuation function can be extended to all  formulas in \textgoth{L} recursively as follows: 
		\begin{align*}
			& V(p) = \pi(p) \;\; \forall p\in \mathcal{P}, & \\
			& V(\neg \varphi) = 1- V(\varphi), & \\
			& V(\varphi \,\&\, \psi) = \max\{0, V(\varphi)+V(\psi) - 1\}, & \\
			& V(\varphi \rightarrow \psi) = \min \{1, 1- V(\varphi)+V(\psi)\}.&
		\end{align*}
We say that a formula $\varphi$ is true and denote it by $\vDash \varphi$, if  for all valuations $V$ we have  $V(\varphi)=1$.
	\end{definition}

	\subsection{Axiomatic system}
	In the following, axiom scheme $(A1)$ to $(A7)$ are the axioms of Basic many-valued Logic (\textbf{BL}) (See \cite{Hajek1998} for more details).
	\begin{equation*}
		\begin{array}{ll}
			(A1) (\varphi \rightarrow \psi) \rightarrow ((\psi \rightarrow \chi) \rightarrow (\varphi \rightarrow \chi)) & (A5) (\varphi \rightarrow (\psi \rightarrow \chi)) \leftrightarrow ((\varphi \, \& \, \psi) \rightarrow \chi) \\
			(A2) (\varphi \, \& \, \psi)\rightarrow \varphi & (A6) ((\varphi \rightarrow \psi)\rightarrow \chi) \rightarrow (((\psi \rightarrow \varphi)\rightarrow \chi)\rightarrow \chi)\\
			(A3) (\varphi \, \& \, \psi) \rightarrow (\psi \, \& \, \varphi) & (A7) \perp \rightarrow \varphi \\ 
			(A4) (\varphi \, \& \, (\varphi \rightarrow \psi)) \rightarrow (\psi \, \& \, (\psi \rightarrow \varphi)) &\\
		\end{array}
	\end{equation*} 
		Propositional \textbf{BL} equipped with the double negation axiom $\neg \neg \varphi \rightarrow \varphi$ is known as propositional \L ukasiewicz logic \textbf{\L}. Traditionally, the following axioms are known as the \L ukasiewicz axioms. 
	\begin{itemize}
		\item[(\L{}1)] $\varphi \rightarrow (\psi \rightarrow \varphi)$
		\item[(\L{}2)] $(\varphi \rightarrow \psi) \rightarrow ((\psi \rightarrow \chi) \rightarrow (\varphi \rightarrow \chi))$
		\item[(\L{}3)] $(\neg \varphi \rightarrow \neg \psi) \rightarrow (\psi \rightarrow \varphi)$
		\item[(\L{}4)] $((\varphi \rightarrow \psi) \rightarrow \psi) \rightarrow ((\psi \rightarrow \varphi) \rightarrow \varphi)$
	\end{itemize}
		 In the following, we give some properties of propositional fuzzy \L ukasiewicz logic which we will use in our proofs. We refer to properties of (\L12), (\L13), and (\L14) as a \textit{replacement} in this paper (See \cite{Hajek1998} for more details).
	$$\begin{array}{llcll}
		(\text{\L}5) & \neg (\varphi \,\&\, \psi) \leftrightarrow (\neg \varphi \veebar \neg \psi) & \quad &
		(\text{\L}6)& \neg (\varphi \veebar \psi) \leftrightarrow (\neg \varphi \, \& \, \neg \psi) \\		
		(\text{\L{}}7)& (\varphi \veebar \psi) \leftrightarrow (\neg \varphi \rightarrow \psi) & \quad&(\text{\L8}) & \neg \neg \varphi \leftrightarrow \varphi \\
		(\L9) &(\varphi \rightarrow \perp) \leftrightarrow \neg \varphi & & (\text{\L{}}10)& (\varphi \,\&\, (\varphi \rightarrow \psi)) \rightarrow \psi \\
(\text{\L{}}11) & 	\multicolumn{4}{l}{((\varphi_1 \rightarrow \psi_1) \,\&\, (\varphi_2 \rightarrow \psi_2))\rightarrow((\varphi_1 \,\&\, \varphi_2)\rightarrow (\psi_1\,\&\, \psi_2))}\\
(\L12 )
& \varphi \rightarrow\varphi \veebar \psi & & \\
		(\text{\L{}}13)& (\varphi \leftrightarrow  \psi) \rightarrow ((\varphi \rightarrow \chi) \leftrightarrow (\psi \rightarrow \chi)) &\quad & &\\
		
		(\text{\L{}}14) &(\varphi \leftrightarrow  \psi) \rightarrow ((\chi \rightarrow \varphi) \leftrightarrow  (\chi \rightarrow \psi))& \quad & & \\
		(\text{\L{}}15) & (\varphi \leftrightarrow  \psi) \rightarrow ((\varphi \,\&\, \chi) \leftrightarrow  (\psi \,\&\, \chi)) & &\\
	\end{array}$$
	\subsection{Completeness} \label{gb_sound_complete}

	\begin{definition}
		 We say that a finite set $\{\varphi_1, \cdots, \varphi_n\}$ is consistent if $\nvdash \neg (\varphi_1\,\&\, \cdots\,\&\, \varphi_n)$. If $\Gamma$ is an infinite set of formulae and all its finite subsets are consistent, then $\Gamma$ is called consistent.
		\\		
		An infinite consistent set of $\Phi$ of formulae in \textgoth{L} is called \textit{maximal} if for all formulae  $\psi\notin\Phi$, the set $\Phi \cup \{\psi\}$  is not consistent.
	
	\end{definition}
	\begin{remark}\label{lemma_consistent}
		Let  $\Phi$ be a set of formulae in \textgoth{L} that is inconsistent. For each set $\Psi$ which $\Phi \subseteq \Psi$  we have $\Psi$ is inconsistent.
	\end{remark}
	\begin{lemma}\label{and_Luka}
		Let $\Gamma = \{\varphi_1, \cdots, \varphi_n\}$ be a set of formulae, then $\Gamma \vdash \varphi_1\,\&\,\cdots\,\&\,\varphi_n$. 
	\end{lemma}
	\begin{proof}
		Without loss of generality, we show that if $\Gamma = \{\varphi_1, \varphi_2\}$, then  $\Gamma \vdash\varphi_1\,\&\,\varphi_2$.
		\begin{align*}
			&(1)&& \Gamma\vdash (\varphi_2\,\&\, \varphi_1)\rightarrow (\varphi_1\,\&\, \varphi_2)&& (A3)\\
			&(2)&& \Gamma\vdash \varphi_2\rightarrow(\varphi_1\rightarrow (\varphi_1\,\&\, \varphi_2)) && \text{(1),(A5), MP}\\
			&(3)&& \Gamma\vdash\varphi_2&& \text{assumption } \varphi_2 \in \Gamma \\
			&(4)&& \Gamma\vdash\varphi_1&& \text{assumption } \varphi_1 \in \Gamma \\
			&(5)&& \Gamma\vdash\varphi_1\rightarrow (\varphi_1\,\&\, \varphi_2)&& \text{(2),(3), MP}\\
			&(6)&&\Gamma\vdash\varphi_1\,\&\, \varphi_2 && \text{(4),(5), MP}
		\end{align*}
	\end{proof}
	\begin{lemma}\label{maximal}
		We have
		\begin{itemize}
			\item[(i)] Each consistent set $\Phi$ of formulae in \textgoth{L} can be extended to a maximal consistent set.
			\item[(ii)] If $\Phi$ is a maximal  consistent set, then for all formulae $\varphi$ and $\psi$:
			\begin{enumerate}
				\item $\varphi\,\&\,\psi \in \Phi$ if and only if $\varphi \in \Phi$ and $\psi\in \Phi$,
				\item If $\varphi \in \Phi$ and $\varphi \rightarrow \psi \in \Phi$, then $\psi \in \Phi$,
				\item If $\Phi\vdash_{ } \varphi$, then $\varphi \in \Phi$,
				\item $\exists n \in \mathbb{N}\; \underbrace{\varphi\,\&\,\cdots\,\&\,\varphi}_{\text{n times}} \in \Phi$ or $\neg \underbrace{(\varphi\,\&\,\cdots\,\&\,\varphi)}_{\text{n times}} \in \Phi$. 
			\end{enumerate}
		\end{itemize}
	\end{lemma}
	\begin{proof}
		\textbf{$(i)$:}	Let $\varphi_1, \varphi_2, \cdots$
		be an enumeration of all formulae in \textgoth{L}. We define a sequence of sets of formulae $\Phi =\Phi_1\subseteq \Phi_2 \subseteq \cdots \subseteq \Phi_i \subseteq \cdots$  as follows:
		\begin{equation*}
			\forall i~~~~\Phi_{i+1} = \left\lbrace \begin{array}{lc}
				\Phi_i & \text{If } \Phi_i\cup \{\varphi_i\} \text{ is not } \text{consistent},\\
				\Phi_i \cup \{\varphi_i\} & \text{otherwise}.
			\end{array} \right. 
		\end{equation*}
		Let $\Phi_\omega = \bigcup_{i\geq 1} \Phi_i$. It is easy to see that $\Phi_\omega$ is consistent since otherwise there is a finite subset $\Gamma\subset \Phi_\omega$ such that it is not consistent, and then for some $j \in \mathbb{N}$ we have $\Gamma \subseteq \Phi_j$, where by Remark \ref{lemma_consistent} $\Phi_j$ is inconsistent but $\Phi_j$ is a consistent set by the definition of construction, which is a contradiction. Therefore $\Phi_\omega$ is consistent. Now suppose that $\Phi_\omega$ is not maximal, so there exists a formula $\varphi=\varphi_j$ such that $\varphi \notin \Phi_\omega$ and the set $\Phi_\omega \cup\{\varphi\}$ is a consistent set. 
		If $\Phi_j\cup \{\varphi_j\}$ is not consistent,  then $\Phi_\omega \cup \{\varphi_j\}$ is not consistent by Remark \ref{lemma_consistent} which is a contradiction. Otherwise $\Phi_{j+1}$  contains $\varphi_j$, and so  $\Phi_\omega$ contains $\varphi_j$ which has a contraction to the assumption $\varphi \notin \Phi_\omega$.
		\vskip 0.3cm
		\textbf{(ii)-1:}  Let $\varphi \, \&\, \psi \in \Phi$ and $\varphi \notin \Phi$. Thus by definition $\Phi \cup \{\varphi\}$ is not consistent. So there is a finite set $\Gamma\subset \Phi$ such that $\Gamma \cup \{\varphi\}$ is not consistent. Let $\Gamma = \{\psi_1, \cdots, \psi_n\}$, hence by the inconsistency of $\Gamma \cup \{\varphi\}$ we have $\vdash_{} \neg (\psi_1 \,\&\, \cdots \,\&\, \psi_n \,\&\, \varphi)$. Since $\Gamma\cup \{\varphi, \psi\}$ is not consistent, then we  have $\vdash_{ } \neg (\psi_1 \,\&\, \cdots \,\&\, \psi_n \,\&\, \varphi \,\&\, \psi)$, and since $\Gamma \cup \{\varphi\,\&\,\psi\}\subset \Phi$, then using Remark \ref{lemma_consistent}, we obtain a contradiction to the  consistency of $\Phi$.\\
		
		For other direction assume $\varphi\in \Phi$, $\psi \in \Phi$ and $\varphi \,\&\, \psi \notin \Phi$. Then $\Phi \cup \{\varphi\,\&\, \psi\} $ is not consistent, and there is a finite set  $\Gamma = \{\psi_1, \cdots, \psi_n\}$, where $\Gamma\subset \Phi$ such that $\vdash_{ }\neg (\psi_1\,\&\, \cdots\,\& \,\psi_n\,\&\, \varphi \,\&\,\psi)$. But this states that the finite subset $\Gamma \cup \{\varphi,\psi\} \subset \Phi$ is not consistent. 
		\\
		\textbf{(ii)-2:}
		Let $\varphi \in \Phi$ and $\varphi \rightarrow \psi \in \Phi$. If $\psi \notin \Phi$, then $\Phi \cup \{\psi\}$ is not consistent and there exists a finite subset $\Gamma \subset \Phi$ such that $\Gamma \cup \{\psi\}$ is not consistent. Let $\Gamma = \{\psi_1, \cdots, \psi_n\}$, so we have:
		\begin{equation}\label{eq_inconsistent}
			\vdash_{ } \neg (\psi_1\,\&\, \cdots, \,\&\, \psi_n \,\&\, \psi).
		\end{equation} 
		On the other hand,
		by using Lemma \ref{and_Luka} from $\vdash_{ } \chi\rightarrow \chi$ and 
		(\L10) we have 
		$\vdash_{ }(\chi \rightarrow \chi) \,\&\, ((\varphi \,\&\, (\varphi \rightarrow \psi)) \rightarrow \psi)$
		and by considering an instance of (\L11) and applying MP we obtain
		$$\vdash_{ }\chi \,\&\, (\varphi \,\&\, (\varphi \rightarrow \psi)) \rightarrow (\chi \,\&\, \psi).$$
		Then by replacing $\psi_1\,\&\, \cdots \,\&\,\psi_n$ instead of $\chi$ we have:
		$$\vdash_{ }(\psi_1\,\&\,\cdots\,\&\, \psi_n \,\&\, \varphi \,\&\,(\varphi \rightarrow \psi)) \rightarrow (\psi_1 \,\&\, \cdots \,\&\, \psi_n \,\& \,\psi )$$
		Now, by using $(\varphi \rightarrow \psi)\rightarrow (\neg \psi \rightarrow \neg \varphi)$ scheme, which is provable in \L ukasiewicz logic we have:
		\begin{equation}\label{eq013} 
			\vdash_{ }\neg (\psi_1 \,\&\, \cdots \,\&\, \psi_n \& \psi ) \rightarrow \neg (\psi_1\,\&\,\cdots\,\&\, \psi_n \,\&\, \varphi \,\&\,(\varphi \rightarrow \psi)).
		\end{equation}
		By applying MP on \ref{eq_inconsistent}, \ref{eq013} we obtain $\vdash_{ } \neg (\psi_1\,\&\,\cdots\,\&\, \psi_n \,\&\, \varphi \,\&\,(\varphi \rightarrow \psi)) $. But this means that $\Gamma \cup \{\varphi, \varphi\rightarrow \psi\}\subset \Phi$ is inconsistent, which is a contradiction to the  -consistency of $\Phi$.\\
		\textbf{(ii)-3:}  We first show that if $\vdash \varphi$, then $\varphi \in \Phi$.
		Let $\vdash \varphi$. 
		First, note that we have $\vdash\neg  \varphi \leftrightarrow \perp$ since from the assumption we have $\vdash \neg\neg \varphi$ from (\L 8) and then from (\L 9) we obtain $\vdash \neg \varphi \rightarrow \perp$, and also from (A7) we have $\vdash \perp \rightarrow \neg \varphi$.
		
	Now, for the sake of contradiction assume that $\varphi \notin \Phi$. If $\varphi \notin \Phi$, then by maximality definition $\Phi\cup \{\varphi\}$ is inconsistent. So there is a subset $\Gamma = \{\psi_1,\cdots, \psi_n\} \subset \Phi$, such that $\Gamma\cup \{\varphi\}$ is not consistent. Therefore 
		$$\begin{array}{llll}
			(1) & & \vdash_{ }\neg (\psi_1\,\&\, \cdots \,\&\, \psi_n \,\&\, \varphi) & \text{inconsistency of } \Gamma \cup \{\varphi\}\\
			(2) & & \vdash_{ }\neg \psi_1 \veebar\cdots \veebar \neg\psi_n \veebar \neg \varphi & (1),(\L5), \text{MP} \\
			(3) & & \vdash_{ }\neg (\neg \psi_1 \veebar\cdots \veebar \neg\psi_n) \rightarrow \neg \varphi & (2), (\L7 ), \text{MP}\\
			(4) & & \vdash_{ }\neg (\neg \psi_1 \veebar\cdots \veebar \neg\psi_n) \rightarrow \perp & (3), \text{assumption and replacement, MP} \\
			(5) & & \vdash_{ }(\neg \neg \psi_1 \,\&\, \cdots \,\&\, \neg \neg \psi_n) \rightarrow \perp & (4), (\L6), \text{MP} \\
			(6)  & & \vdash_{ }(\psi_1 \,\&\, \cdots \,\&\, \psi_n) \rightarrow \perp & (5), (\L8), \text{replacement}, \text{MP}\\
			(7) & & \vdash_{ }\neg (\psi_1 \,\&\, \cdots \,\&\,\psi_n) & (6),(\L9), \text{MP}
		\end{array}$$
		Thus from $\vdash_{ } \neg (\psi_1 \,\&\, \cdots \,\&\,\psi_n)$, we have $\Gamma$ is not consistent. But this is a contradiction with the consistency of $\Phi$.
		
		And by induction on length of proof of $\varphi$ from $\Phi$, and the previous fact and the part (ii)-2, the desired statement can be obtained.
		
		\textbf{(ii)-4:} Suppose that there is a formula $\varphi$ in \textgoth{L} such that for every $n\in \mathbb{N}$  $\underbrace{\varphi\,\&\,\cdots\,\&\,\varphi}_{\text{n times}}\notin \Phi$ and $\neg \underbrace{(\varphi\,\&\,\cdots\,\&\,\varphi)}_{\text{n times}} \notin \Phi$. From the maximality and consistency of $\Phi$, for all $n\in\mathbb{N}$ neither  $\Phi \cup \{\underbrace{\varphi\,\&\,\cdots\,\&\,\varphi}_{\text{n times}}\}$ nor $\Phi \cup \{\neg \underbrace{(\varphi\,\&\,\cdots\,\&\,\varphi)}_{\text{n times}}\}$ is consistent.
		Assume  $\Gamma = \{\psi_1, \cdots, \psi_m\}$ is a finite subset of $\Phi$ which  $\Gamma \cup \{\underbrace{\varphi\,\&\,\cdots\,\&\,\varphi}_{\text{n times}}\}$ is not consistent,
		then the following deduction is valid:
		$$\begin{array}{llll}
			(1)& &\vdash_{ } \neg (\psi_1 \,\&\, \cdots \,\&\, \psi_m \,\&\,\underbrace{(\varphi\,\&\,\cdots\,\&\,\varphi)}_{\text{n times}}) & \text{inconsistency of } \Gamma\cup \{\underbrace{\varphi\,\&\,\cdots\,\&\,\varphi}_{\text{n times}}\}\\
			(2)& & \vdash_{ } \neg \psi_1 \veebar \cdots \veebar \neg \psi_m \veebar \neg \underbrace{(\varphi\,\&\,\cdots\,\&\,\varphi)}_{\text{n times}} & (1), (\text{\L} 5)\\ 
			(3)& &\vdash_{ } \neg (\neg \psi_1 \veebar \cdots \veebar \neg \psi_m ) \rightarrow \neg \underbrace{(\varphi\,\&\,\cdots\,\&\,\varphi)}_{\text{n times}} & (2), (\text{\L}7) \\
			(4)& &\vdash_{ } (\neg\neg \psi_1 \,\&\, \cdots \,\&\,\neg\neg\psi_m ) \rightarrow \neg \underbrace{(\varphi\,\&\,\cdots\,\&\,\varphi)}_{\text{n times}}  & (3), (\text{\L}6) \\
			(5)& & \vdash_{ } (\psi_1 \,\&\, \cdots \,\&\,\psi_m ) \rightarrow \neg \underbrace{(\varphi\,\&\,\cdots\,\&\,\varphi)}_{\text{n times}} & (4), (\text{\L}8), \text{replacement}.
		\end{array}$$
		Thus from (ii)-3 we have $(\psi_1 \,\&\, \cdots \,\&\, \psi_m )\rightarrow \neg \underbrace{(\varphi\,\&\,\cdots\,\&\,\varphi)}_{\text{n times}} \in \Phi$,
		since $\Gamma=\{\psi_1, \cdots, \psi_m\}\subset \Phi$ then we have $(\psi _1 \,\&\, \cdots \,\&\, \psi_m) \in \Phi$, by (ii)-1. So from (ii)-2 we have $\neg \underbrace{(\varphi\,\&\,\cdots\,\&\,\varphi)}_{\text{n times}} \in \Phi$, which is a  contradiction with the assumption  $\neg \underbrace{(\varphi\,\&\,\cdots\,\&\,\varphi)}_{\text{n times}} \notin \Phi$.
		
%
	\end{proof}
		\begin{remark}
		Unlike  classical maximal consistent sets, here the maximal sets do not necessarily include all formulae or their negations. For example consider the set $\{(p\,\&\,p), \neg (\neg p \,\&\, \neg p)\}\subseteq \Phi$. This maximal consistent set  does not include neither $p$ nor $\neg p$. However, note that $(p\,\&\,p)\in \Phi$ which justifies Lemma \ref{maximal} part ii-4  for $n=2$.
	\end{remark}

		\begin{theorem}\label{consistent_e}
			Let $\Phi$ be a consistent set of formulae and $\varphi$ be a formula such that 
			$\Phi \nvdash \varphi$. If $\Phi^{*} = \Phi\cup \{\neg \varphi\}$, then $\Phi^{*}$ is consistent.
		\end{theorem}
		\begin{proof}
			For the sake of contradiction suppose that $\Phi^{*}$ is not consistent. So there exists a finite subset $\Gamma = \{\varphi_1, \cdots, \varphi_n\}\subseteq\Phi$ such that $\Gamma \cup \{\neg \varphi\}$ is not consistent, then we have:
			\begin{align}
				\nonumber	& \vdash \neg (\varphi_1\,\&\, \cdots \,\&\, \varphi_n \,\&\, \neg \varphi)  &  \\
				\nonumber		& \vdash \neg \varphi_1 \veebar \cdots \veebar \neg \varphi_n \veebar \neg \neg \varphi & \\
				\nonumber		& \vdash \neg (\neg \varphi_1 \veebar \cdots \veebar \neg \varphi_n) \rightarrow \varphi & \\
				\label{eq_008}		& \vdash (\varphi_1 \,\&\, \cdots \,\&\, \varphi_n) \rightarrow \varphi &
			\end{align}
			Also, by Lemma \ref{and_Luka} part 2-i, from $\Gamma\subset \Phi$ we have $\Phi \vdash_{ } (\varphi_1 \,\&\, \cdots \,\&\, \varphi_n) $.
			Thus from \ref{eq_008}, we have $\Phi \vdash_{ } \varphi$,
			which is a contradiction to $\Phi \nvdash_{ } \varphi$.
			
		\end{proof}
		
		\begin{theorem} \label{model_e}
			Let $\Phi$ be a consistent set and $\Phi \vdash_{ } \varphi$, where $\varphi$ is a formula. Then there is a valuation $V$  such that $V(\varphi) = 1$. 
		\end{theorem}
		\begin{proof}
				First, we define a valuation $V$, then it is enough to show that for an arbitrary maximal   $\Phi^*$ contains $\Phi$, for each $\varphi'\in \Phi^*$,  we have the following:
				\begin{equation} \label{eq014}
					V(\varphi') = \left \lbrace\begin{array}{lll}
						1 & \iff & \varphi' \in \Phi^* \\
						0 & \iff & \neg \varphi' \in \Phi^* \\ 
						0.5 & \iff & \text{otherwise  }  \\
					\end{array} \right.
				\end{equation}
			It is easy to check the case $\varphi'=\varphi$,  the theorem statement holds.
			The valuation $V$ for each proposition $p$  is defined as follows:
				\begin{align*}
					& V( p) = \left\lbrace \begin{array}{lc}
						1 & p\in \Phi^* \\
						0 & \neg p \in \Phi^* \\
						0.5 & \text{otherwise } \
					\end{array} \right.; \qquad p\in \mathcal{P}.
				\end{align*}
				We prove statement \ref{eq014} by induction on the complexity of $\varphi'$.
				It is not hard to check that the base step $\varphi'=p$ holds by the definition of $V$.\\
				 For the induction step, we have the following cases:
				\\
				\textbf{case 1:}  $\varphi' = \neg \psi$. If  $\neg \psi \in \Phi^*$, then by consistency of $\Phi^*$ we have $\psi \notin \Phi^*$.
				So by induction hypothesis on $\psi$ we have $V(\psi)=0$. Thus $V(\neg \psi)=1$.
				
				Now assume $\neg \psi \notin \Phi^*$. If $\psi \notin \Phi^*$, then by induction on $\psi$ we have $V(\psi) = 0.5$ and so $V(\neg \psi) = 0.5$. If $\psi \in \Phi^*$, by induction on $\psi$ we have $V(\psi) = 1$ and so $V(\neg \psi)=0$ by definition.
				
				For other direction if $V(\neg \psi) = 1$, then by definition we have $V(\psi) = 0$. So by induction hypothesis on $\psi$ we have $\neg \psi\in \Phi^*$. 

				\textbf{case 2:} $\varphi' = \psi \,\&\, \chi$. If we have $\psi \,\&\, \chi \in \Phi^*$ then
				$$\psi \,\&\, \chi \in \Phi^* \stackrel{\text{Lemma } \ref{maximal}}{\iff} \psi \in \Phi^*, \chi \in \Phi^*\stackrel{\text{induction}}{\iff} V (\psi)=1= V(\chi) \Leftrightarrow V(\psi \,\&\, \chi)=1. $$
				Similar argument can be applied when $\psi\,\&\,\chi \notin \Phi^*$.\\
				\textbf{case 3:} $\varphi' = \psi \rightarrow \chi$.
				First assume $V(\psi \rightarrow \chi) = 1$. So we have $V(\psi) \leq V(\chi)$. It is not hard to check that by definition all values belongs to the set $\{0,0.5,1\}$ and so there are six possible cases in which $V(\psi) \leq V(\chi)$. If $V(\chi)=1$, then by induction we have $\chi \in \Phi^*$ and using (\L1) and MP we obtain $\psi\rightarrow\chi\in \Phi^*$. If $V(\psi)=0$, then by induction we have $\neg \psi \in \Phi^*$ and using (\L1) as $\neg \psi \rightarrow (\neg \chi \rightarrow \neg \psi)$ and MP we obtain $\neg \chi \rightarrow \neg \psi$, then by (\L3) and MP we obtain $\psi\rightarrow \chi\in \Phi^*$. If $V(\psi) = V(\chi) = 0.5$, then by induction hypothesis we have $\psi,\neg \chi \notin \Phi^*$. By Lemma \ref{maximal}, $\psi\,\&\,\neg \chi \notin \Phi^*$ and so we can easily obtain $\neg (\neg \psi \veebar \neg\neg \chi) \notin \Phi^*$. Using Theorem \ref{consistent_e} the set $\Phi^*\cup \{\neg\neg (\neg \psi \veebar \neg\neg \chi)\}$ is consistent. Using $(\L8)$ we can conclude that $\Phi^*\cup \{\neg \psi \veebar \neg\neg \chi\}$ is consistent. From (\L7) and replacement $\Phi^*\cup \{\psi\rightarrow \chi\}$ is consistent and since $\Phi^*$ is maximal we can conclude that $\psi\rightarrow\chi \in \Phi^*$.

				For the other direction assume $\psi \rightarrow \chi \in \Phi^*$. The desired statement can be concluded easily when $\chi \in \Phi^*$ by induction hypothesis.
				So assume $\chi \notin \Phi^*$. Note that if $\psi \in \Phi^*$ by Lemma \ref{maximal}-ii-2 we have $\chi \in \Phi^*$ which is contradiction with assumption. If $\neg \psi\in \Phi^*$ by induction hypothesis on $\psi$ we have $V(\psi)=0$ and so $V(\psi\rightarrow \chi)=1$. If $\neg \psi\notin \Phi^*$, then by induction hypothesis on $\psi$ and $\chi$, if $\neg \chi \notin \Phi^*$, we obtain $V(\psi\rightarrow \chi)=1$. Note that it is not possible that $\neg \chi \in \Phi^*$ since otherwise by using Lemma \ref{maximal}, and (\L 3) as $(\psi\rightarrow\chi) \rightarrow (\neg \chi\rightarrow \neg \psi)$ and applying MP two times we obtain $\neg \psi\in \Phi^*$ which contradicts with our assumption.
				
				If $\psi\rightarrow\chi \notin \Phi^*$, then $\Phi^*\cup \{\psi\rightarrow\chi\}$ is inconsistent. So there is a finite subset $\Gamma = \{\psi_1, \cdots, \psi_n\}$ of $\Phi^*$ such that $\Gamma \cup \{\psi\rightarrow\chi\}$ is inconsistent. Thus by following argument 				
				$$\begin{array}{llll}
					(1)& &\vdash_{ } \neg (\psi_1 \,\&\, \cdots \,\&\, \psi_m \,\&\,(\psi\rightarrow\chi)) & \text{inconsistency of } \Gamma\cup \{(\psi\rightarrow\chi)\}\\
					(2)& & \vdash_{ } \neg \psi_1 \veebar \cdots \veebar \neg \psi_m \veebar \neg (\psi\rightarrow\chi) & (1), (\text{\L} 5)\\ 
					(3)& &\vdash_{ } \neg (\neg \psi_1 \veebar \cdots \veebar \neg \psi_m ) \rightarrow \neg (\psi\rightarrow\chi) & (2), (\text{\L}7) \\
					(4)& &\vdash_{ } (\neg\neg \psi_1 \,\&\, \cdots \,\&\,\neg\neg\psi_m ) \rightarrow \neg (\psi\rightarrow\chi) & (3), (\text{\L}6) \\
					(5)& & \vdash_{ } (\psi_1 \,\&\, \cdots \,\&\,\psi_m ) \rightarrow \neg (\psi\rightarrow\chi)& (4), (\text{\L}8), \text{replacement}.
				\end{array}$$
				and using Remark \ref{and_Luka} and Lemma \ref{maximal}-ii-2 we have $\neg (\psi\rightarrow\chi)\in \Phi^*$.
				Applying (\L 6) and (\L 7) by Lemma \ref{maximal}-ii-2 we obtain $\psi \,\&\,\neg \chi\in \Phi^*$ and so $\psi,\neg \chi \in \Phi^*$. Therefore by induction hypothesis we have $V(\psi) = 1$ and $V(\chi)=0$ which means $V(\psi\rightarrow\chi)=0$ as desired. 
				\\
			\end{proof}
%
		
			\begin{theorem} \label{thm00} \textbf{(Completeness)}
				If $ \vDash\varphi$, then $\vdash_{ } \varphi$.
			\end{theorem}
			\begin{proof}
				For the sake of contradiction assume we have $\nvdash_{ } \varphi$. Thus by Theorem \ref{consistent_e}, $ \{\neg \varphi\}$ is a consistent set, and by Lemma \ref{maximal} we have a maximal and consistent set $\Phi^*$ which contains $\{\neg \varphi\}$. Therefore there exists valuation function $V$ such that $V(\neg \varphi)=1$ by Theorem \ref{model_e}. But this is a contradiction to $\vDash \varphi$. 
			\end{proof}

		\end{document}